\documentclass[a4paper,10pt, reqno]{amsart}
\usepackage{amssymb,latexsym,amsmath,epsfig,amsthm} 

\theoremstyle{plain}
\newtheorem{theorem}{Theorem}
\newtheorem*{theorem*}{Theorem}
\newtheorem{corollary}{Corollary}
\newtheorem*{corollary*}{Corollary}
\newtheorem{lemma}{Lemma}
\newtheorem*{lemma*}{Lemma}

\newtheorem*{proposition*}{Proposition}

\newtheorem*{conjecture*}{Conjecture}
\theoremstyle{definition}

\newtheorem*{definition*}{Definition}
\theoremstyle{remark}
\newtheorem{remark}{Remark}
\newtheorem*{remark*}{Remark}
\theoremstyle{example}
\newtheorem{example}{Example}
\newtheorem*{example*}{Example}

\begin{document}
\title[Cantor series and functions]{On certain maps defined by infinite sums}

\author{Symon Serbenyuk}
\address{ 
  45~Shchukina St. \\
  Vinnytsia \\
 21012 \\
  Ukraine}
\email{simon6@ukr.net}



\begin{abstract}

The present article is devoted to some examples of functions whose arguments represented in terms of certain series of the Cantor type. 

\end{abstract}
\maketitle



\section{Introduction}

Let $Q\equiv (q_k)$ be a fixed sequence of positive integers, $q_k>1$,  $\Theta_k$ be a sequence of the sets $\Theta_k\equiv\{0,1,\dots ,q_k-1\}$, and $\varepsilon_k\in\Theta_k$.

Real number expansions of the form
\begin{equation}
\label{eq: Cantor series}
\frac{\varepsilon_1}{q_1}+\frac{\varepsilon_2}{q_1q_2}+\dots +\frac{\varepsilon_k}{q_1q_2\dots q_k}+\dots
\end{equation}
for $x\in [0,1]$,   first studied by G. Cantor in \cite{C1869}. It is easy to see that the last expansion is the  $q$-ary expansion
$$
\frac{\alpha_1}{q}+\frac{\alpha_2}{q^2}+\dots+\frac{\alpha_n}{q^n}+\dots
$$
of numbers  from  $[0,1]$ whenever the condition $q_k=q$ holds for all positive integers $k$. Here $q$ is a fixed positive integer, $q>1$, and $\alpha_n\in\{0,1,\dots , q-1\}$.

By $x=\Delta^Q _{\varepsilon_1\varepsilon_2\ldots\varepsilon_k\ldots}$  denote a number $x\in [0,1]$ represented by series \eqref{eq: Cantor series}. This notation is called \emph{the representation of $x$ by (positive) Cantor series \eqref{eq: Cantor series}.}

Let us remark that certain numbers from $[0,1]$ have two different representations by positive Cantor series \eqref{eq: Cantor series}, i.e., 
$$
\Delta^Q _{\varepsilon_1\varepsilon_2\ldots\varepsilon_{m-1}\varepsilon_m000\ldots}=\Delta^Q _{\varepsilon_1\varepsilon_2\ldots\varepsilon_{m-1}[\varepsilon_m-1][q_{m+1}-1][q_{m+2}-1]\ldots}=\sum^{m} _{i=1}{\frac{\varepsilon_i}{q_1q_2\dots q_i}}.
$$
Such numbers are called \emph{$Q$-rational}. The other numbers in $[0,1]$ are called \emph{$Q$-irrational}.

Now a number of researchers introduce and/or study alternating versions (types) of well-known positive expansions. For example, 
investigations of positive and alternating L\"uroth series and Engel series (e.g., see~\cite{Fang2015, KKK90, L1883, Engel}), as well as of $\beta$- and $(-\beta)$-expansions (\cite{IS2009, Renyi}) are such researhes. 

Since investigations for the cases of alternating expansions require more complicated techniques, let us consider functions whose arguments defined in terms of alternating series of the Cantor type. The present investigations are similar with investigations (\cite{Serbenyuk 2019 functions}) for positive Cantor series  but are more complicated. 

 In~\cite{S. Serbenyuk alternating Cantor series 2013},  the following expansions of real numbers were studied:
\begin{equation}
\label{eq: alternating Cantor series}
-\frac{\varepsilon_1}{q_1}+\frac{\varepsilon_2}{q_1q_2}-\frac{\varepsilon_3}{q_1q_2q_3}+\dots +\frac{(-1)^k\varepsilon_k}{q_1q_2\dots q_k}+\dots
\end{equation}
for $x\in [a_0-1,a_0]$,  where $a_0=\sum^{\infty} _{k=1}{\frac{q_{2k}-1}{q_1q_2\cdots q_{2k}}}$.  Here $-Q\equiv (-q_k)$ is a fixed sequence of negative integers $(-q_k)<-1$,  $\Theta_k\equiv\{0,1,\dots ,q_k-1\}$, and $\varepsilon_k\in\Theta_k$.

It is easy to see that the last expansion is the  nega-$q$-ary expansion
\begin{equation}
\label{eq: nega-q}
\Delta^{-q} _{\alpha_1\alpha_2\alpha_3...\alpha_k...}\equiv -\frac{\alpha_1}{q}+\frac{\alpha_2}{q^2}-\frac{\alpha_3}{q^3}+\dots+\frac{(-1)^k\alpha_k}{q^k}+\dots
\end{equation}
of numbers  from  $\left[-\frac{q}{q+1},\frac{1}{q+1}\right]$ whenever the condition $q_k=q$ holds for all positive integers $k$. Here $q$ is a fixed positive integer, $q>1$, and $\alpha_n\in\{0,1,\dots , q-1\}$. 

By $x=\Delta^{-Q} _{\varepsilon_1\varepsilon_2\ldots\varepsilon_k\ldots}$  denote a number $x\in [a_0-1,a_0]$ represented by series \eqref{eq: alternating Cantor series}. This notation is called \emph{the representation of $x$ by alternating Cantor series \eqref{eq: Cantor series} or the nega-$Q$-representation.}

The term ``nega" is used in this article, since the alternating Cantor series expansion is a numeral system with a negative base $(-q_k)$.

Some numbers  have two different representations by alternating  series \eqref{eq: alternating Cantor series}, i.e., 
$$
\Delta^{-Q} _{\varepsilon_1\varepsilon_2\ldots\varepsilon_{m-1}\varepsilon_m[q_{m+1}-1]0[q_{m+3}-1]0[q_{m+5}-1]\ldots}=\Delta^{-Q} _{\varepsilon_1\varepsilon_2\ldots\varepsilon_{m-1}[\varepsilon_m-1]0[q_{m+2}-1]0[q_{m+4}-1]0[q_{m+6}-1]\ldots}.
$$
Such numbers are called \emph{nega-$Q$-rational}. The other numbers in $[a_0-1,a_0]$ are called \emph{nega-$Q$-irrational}.

Suppose $c_1,c_2,\dots, c_m$ is  an
ordered tuple of integers such that $c_i\in\{0,1,\dots, q_i-~1\}$ for $i=\overline{1,m}$. Then \emph{a cylinder $\Delta^{-Q} _{c_1c_2...c_m}$ of rank $m$ with base $c_1c_2\ldots c_m$} is a set of the form
$$
\Delta^{-Q} _{c_1c_2...c_m}\equiv\{x: x=\Delta^{-Q} _{c_1c_2...c_m\varepsilon_{m+1}\varepsilon_{m+2}\ldots\varepsilon_{m+k}\ldots}\}.
$$
That is any cylinder $\Delta^{-Q} _{c_1c_2...c_m}$ is a closed interval of the form:
$$
\left[\Delta^{-Q} _{c_1c_2...c_m[q_{m+1}-1]0[q_{m+3}-1]0[q_{m+5}-1]...}, \Delta^{-Q} _{c_1c_2...c_m0[q_{m+2}-1]0[q_{m+4}-1]0[q_{m+6}-1]...}\right]~~~\text{if $m$ is even},
$$
$$
\left[\Delta^{-Q} _{c_1c_2...c_m0[q_{m+2}-1]0[q_{m+4}-1]0[q_{m+6}-1]...},\Delta^{-Q} _{c_1c_2...c_m[q_{m+1}-1]0[q_{m+3}-1]0[q_{m+5}-1]...}\right]~~~\text{if $m$ is odd}.
$$

Define \emph{the shift operator $\sigma$ of expansion \eqref{eq: alternating Cantor series}} by the rule
$$
\sigma(x)=\sigma\left(\Delta^{-Q} _{\varepsilon_1\varepsilon_2\ldots\varepsilon_k\ldots}\right)=\sum^{\infty} _{k=2}{\frac{(-1)^k\varepsilon_k}{q_2q_3\dots q_k}}=-q_1\Delta^{Q} _{0\varepsilon_2\ldots\varepsilon_k\ldots}.
$$
Whence,
\begin{equation}
\label{eq: Cantor series 2}
\begin{split}
\sigma^n(x) &=\sigma^n\left(\Delta^{-Q} _{\varepsilon_1\varepsilon_2\ldots\varepsilon_k\ldots}\right)\\
& =\sum^{\infty} _{k=n+1}{\frac{(-1)^{k-n}\varepsilon_k}{q_{n+1}q_{n+2}\dots q_k}}=(-1)^nq_1\dots q_n\Delta^{-Q} _{\underbrace{0\ldots 0}_{n}\varepsilon_{n+1}\varepsilon_{n+2}\ldots}.
\end{split}
\end{equation}
Therefore, 
\begin{equation}
\label{eq: Cantor series 3}
x=\sum^{n} _{i=1}{\frac{(-1)^i\varepsilon_i}{q_1q_2\dots q_i}}+\frac{(-1)^n}{q_1q_2\dots q_n}\sigma^n(x).
\end{equation}
 The notion of the shift operator of an alternating Cantor series was studied in detail  in the paper \cite{S. Serbenyuk alternating Cantor series 2013}.

In \cite{Salem1943}, the following singular function 
$$
s(x)=s\left(\Delta^2 _{\alpha_1\alpha_2...\alpha_n...}\right)=\beta_{\alpha_1}+ \sum^{\infty} _{n=2} {\left(\beta_{\alpha_n}\prod^{n-1} _{i=1}{q_i}\right)}=y=\Delta^{Q_2} _{\alpha_1\alpha_2...\alpha_n...},
$$
where $q_0>0$, $q_1>0$, and $q_0+q_1=1$, was modeled by Salem. Note that generalizations of the Salem function can be non-differentiable functions or do not have a derivative on a certain set. 

Let us consider the following generalizations of the Salem function that are described in the paper \cite{S. Serbenyuk systemy rivnyan 2-2} as well. 

\begin{example}[\cite{Symon2015}]
\label{example: 1}
{\rm Let $(q_n)$ is a fixed sequence of positive integers, $q_n>1$, and $(A_n)$ is a sequence of the sets  $\Theta_n = \{0,1,\dots,q_n-1\}$.

Let $x\in [0,1]$ be an arbitrary number represented by a positive Cantor series
\begin{equation*}
x=\Delta^Q _{\varepsilon_1\varepsilon_2...\varepsilon_n...}=\sum^{\infty} _{n=1} {\frac{\varepsilon_n}{q_1q_2\dots q_n}}, ~\mbox{where}~\varepsilon_n \in \Theta_n.
\end{equation*}

Let $P=||p_{i,n}||$  be a fixed matrix such that  $p_{i,n}\in (-1,1)$ ($ n=1,2,\dots ,$ and $i=~\overline{0,q_n-1}$), $\sum^{q_n-1} _{i=0} {p_{i,n}}=1$ for an arbitrary $n \in \mathbb N$, and $\prod^{\infty} _{n=1}{p_{i_n,n}}=0$ for any sequence   $(i_n)$.

Suppose that elements of the matrix $P=||p_{i,n,}||$ can be negative numbers as well but   
$$
\beta_{0,n}=0, \beta_{i,n}>0 ~\mbox{for}~ i\ne 0, ~\mbox{and} ~ \max_i {|p_{i,n}|} <1.
$$
Here 
$$
\beta_{\varepsilon_{k},k}=\begin{cases}
0&\text{if $\varepsilon_{k}=0$}\\
\sum^{\varepsilon_{k}-1} _{i=0} {p_{i,k}}&\text{if $\varepsilon_{k}\ne 0$.}
\end{cases}
$$
Then the following statement is true.
\begin{theorem}[\cite{Symon2015}]
Given the matrix $P$  such that for all $n \in \mathbb N$ the following are true:  $p_{\varepsilon_n,n}\cdot p_{\varepsilon_n-1,n}<0$ moreover $q_n \cdot~p_{d_n-1,n}\ge 1$ or  $q_n \cdot p_{q_n-1,n}\le 1$; and the  conditions 
$$
\lim_{n \to \infty} {\prod^{n} _{k=1} {q_k p_{0,k}}}\ne  0, \lim_{n \to \infty} {\prod^{n} _{k=1} {q_k p_{q_k-1,k}}}\ne 0
$$
hold simultaneously.
Then the function  
$$
F(x)=\beta_{\varepsilon_1(x),1}+\sum^{\infty} _{k=2} {\left(\beta_{\varepsilon_k(x),k}\prod^{k-1} _{n=1} {p_{\varepsilon_n(x),n}}\right)}
$$
is non-differentiable on  $[0,1]$.
\end{theorem}
} 
\end{example}
\begin{example}[\cite{Symon2017}]
\label{ex: example 2}
{\rm
Let
$P=||p_{i,n}||$ be a given matrix such that  $n=1,2, \dots$ and $i=\overline{0,q_n-1}$. For this matrix the following system of properties  holds: 
$$ 
\left\{
\begin{aligned}
\label{eq: tilde Q 1}
1^{\circ}.~~~~~~~~~~~~~~~~~~~~~~~~~~~~~~~~~~~~~~~~~~~~~~~\forall n \in \mathbb N:  p_{i,n}\in (-1,1)\\
2^{\circ}.  ~~~~~~~~~~~~~~~~~~~~~~~~~~~~~~~~~~~~~~~~~~~~~~~~\forall n \in \mathbb N: \sum^{q_n-1}_{i=0} {p_{i,n}}=1\\
3^{\circ}. ~~~~~~~~~~~~~~~~~~~~~~~~~~~~~~~~~~~~~ \forall (i_n), i_n \in  \Theta_{n}: \prod^{\infty} _{n=1} {|p_{i_n,n}|}=0\\
4^{\circ}.~~~~~~~~~~~~~~\forall  i_n \in \Theta_{n}\setminus\{0\}: 1>\beta_{i_n,n}=\sum^{i_n-1} _{i=0} {p_{i,n}}>\beta_{0,n}=0.\\
\end{aligned}
\right.
$$ 

 Let us consider the following function 
$$ 
\tilde{F}(x)=\beta_{\varepsilon_1(x),1}+\sum^{\infty} _{n=2} {\left(\tilde{\beta}_{\varepsilon_n(x),n}\prod^{n-1} _{j=1} {\tilde{p}_{\varepsilon_j(x),j}}\right)},
$$
where
$$
\tilde{\beta}_{\varepsilon_n(x),n}=\begin{cases}
\beta_{\varepsilon_n(x),n}&\text{if $n$ is   odd }\\
\beta_{q_n-1-\varepsilon_n(x),n}&\text{if $n$ is  even,}
\end{cases}
$$
$$
\tilde{p}_{\varepsilon_n(x),n}=\begin{cases}
p_{\varepsilon_n(x),n}&\text{if $n$  is odd }\\
p_{q_n-1-\varepsilon_n(x),n}&\text{if $n$  is   even,}
\end{cases}
$$
$$
\beta_{\varepsilon_{n}(x),n}=\begin{cases}
0&\text{if $\varepsilon_{n}=0$}\\
\sum^{\varepsilon_{n}-1} _{i=0} {p_{i,n}}&\text{if $\varepsilon_{n}\ne 0$.}
\end{cases}
$$
Here $x$ represented by an alternating Cantor series, i.e., 
$$
x=\Delta^{-(q_n)} _{\varepsilon_1\varepsilon_2...\varepsilon_n...}=\sum^{\infty} _{n=1} {\frac{1+\varepsilon_n}{q_1q_2\dots q_n}(-1)^{n+1}},
$$
where $(q_n)$ is a fixed sequence of positive integers, $q_n>1$, and $(\Theta_{n})$ is a sequence of the sets  $\Theta_{n} = \{0,1,\dots,q_n-1\}$, and $\varepsilon_n\in \Theta_{n}$.
\begin{theorem}
Let  $p_{\varepsilon_n,n}\cdot p_{\varepsilon_n-1,n}<0$  for all $n \in \mathbb N$, $\varepsilon_n \in \Theta_{n} \setminus \{0\}$ and conditions 
$$
\lim_{n \to \infty} {\prod^{n} _{k=1} {q_k p_{0,k}}}\ne  0, \lim_{n \to \infty} {\prod^{n} _{k=1} {q_k p_{q_k-1,k}}}\ne 0
$$
hold simultaneously.  Then the function $\tilde{F}$ is  non-differentiable on $[0,1]$. 
\end{theorem}
 }
\end{example}

In the present article, two examples  of certain functions with complicated local structure, are constructed and investigated. 

Suppose  that the condition $q_n\le q$ holds for all positive integers $n$.
The first function is following:
$$
f: ~~~ x=\Delta^{-Q} _{\varepsilon_1\varepsilon_2... \varepsilon_n...} ~\longrightarrow~\Delta^{-q} _{\varepsilon_1\varepsilon_2... \varepsilon_n...}=y.
$$

The functon $f$ is interesting, since a some function, whose almost all properties (all properties without the domain of definition; the domains of definition of these functions are different intervals) are identical with properties of  the function described in Example~\ref{ex: example 2}, can be represented by the following way:
$$
F(x)=\ddot F_{\xi, Q} \circ g \circ f.
$$ 
 Here by ``$\circ$" denote the operation of composition of functions and $g(x)=x-\Delta^{-q} _{([q-1]0)}$. Also, the function $\ddot F_{\xi, Q}$ is a function of the type:
$$
\ddot F_{\eta, Q}(y)=\ddot\beta_{\varepsilon_1(y),1}+\sum^{\infty} _{k=2} {\left({\ddot\beta}_{\varepsilon_k(y),k} \prod^{k-1} _{j=1} {{\ddot p}_{\varepsilon_j(y),j}}\right)},
$$
where 
$$
y=\Delta^{-q} _{\varepsilon_1\varepsilon_2... \varepsilon_n...}-\Delta^q _{[q_1-1]0[q_3-1]0...[q_{2k-1}-1]0...}=\Delta^{-q} _{[q_1-1-\varepsilon_1]0[q_3-1-\varepsilon_3]0...[q_{2k-1}-1-\varepsilon_{2k-1}]0...}.
$$

Note that the function $\ddot F_{\eta, q}$ is a distribution function of a certain random variable $\eta$ whenever elements $ p_{i,n}$ of the matrix $P$ (this matrix described in the last-mentioned examples) are non-negative and 
$$
\ddot{p}_{\varepsilon_n(x),n}=\begin{cases}
p_{\varepsilon_n(x),n}&\text{if $n$  is even }\\
p_{q_n-1-\varepsilon_n(x),n}&\text{if $n$  is  odd},
\end{cases}
$$
$$
\ddot{\beta}_{\varepsilon_n(x),n}=\begin{cases}
\beta_{\varepsilon_n(x),n}&\text{if $n$  is even }\\
\beta_{q_n-1-\varepsilon_n(x),n}&\text{if $n$  is  odd}.
\end{cases}
$$

\begin{remark}{\cite{Serbenyuk 2019 functions}.}
``In the general case, suppose that $(f_n)$ is a finite or infinite sequence of certain functions (the sequence can contain functions with complicated local structure). Let us consider the corresponding composition of the functions
$$
\ldots \circ f_n \circ \ldots \circ f_2 \circ f_1=f_{c,\infty}
$$
or
$$
 f_n \circ \ldots \circ f_2 \circ f_1=f_{c,n}.
$$
Also, we can take a certain part of the composition, i.e.,
$$
 f_{n_0+t} \circ \ldots \circ f_{n_0+1} \circ f_{n_0}=f_{c,\overline{n_0,{n_0+t}}},
$$
where $n_0$ is a fixed positive integer (a number from the set $\mathbb N$), $t\in \mathbb Z_0=\mathbb N \cup\{0\}$, and $n_0+t \le n$.

One can use such technique for modeling and studying  functions with complicated local structure. Also, one can use  new representations of real numbers (numeral systems) of the type
$$
x^{'}=\Delta^{f_{c,\infty}} _{i_1i_2...i_n}=\ldots \circ f_n \circ \ldots \circ f_2 \circ f_1(x), 
$$
$$
x^{'}=\Delta^{f_{c,n}} _{i_1i_2...i_n}=f_n \circ \ldots \circ f_2 \circ f_1(x)
$$
or
$$
z^{'}=\Delta^{f_{c,\overline{n_0,{n_0+t}}}} _{i_1i_2...i_n}=f_{n_0+t} \circ \ldots \circ f_{n_0+1} \circ f_{n_0}(z).
$$
in fractal theory, applied   mathematics, etc. The next articles of the author of the present article will be devoted to such investigations".
\end{remark}

\begin{remark}
One can extend the last remark by the following. Really, compositions of functions are useful for modeling functions with complicated local structure. However, for modeling such functions one can use systems of functional equation containing compositions of functions. For example, 
$$
f\left(\underbrace{g\circ g \circ \ldots \circ g (x)}_{k-1}\right)=a_k+b_kf\left(\underbrace{g\circ g \circ \ldots \circ g (x)}_{k}\right),
$$
where $f,g$ are some functions, $a_k,b_k \in\mathbb R$.

For example, in~\cite{Serbenyuk2019}, a technique for modeling certain generalizations of the singular Salem function is introduced. That is, 
$$
f\left(\sigma_{n_{k-1}}\circ \sigma_{n_{k-2}}\circ \ldots \circ \sigma_{n_1}(x)\right)=\beta_{\alpha_{n_k}, n_k}+p_{\alpha_{n_k}, n_k}f\left(\sigma_{n_{k}}\circ \sigma_{n_{k-1}}\circ \ldots \circ \sigma_{n_1}(x)\right),
$$
where $k=1,2,\dots$, $\sigma_0(x)=x$, and $x$ represented in terms of a certan given numeral system, i.e., $x=\Delta_{\alpha_1\alpha_2...\alpha_k...}$ and $\alpha_n\in\{0,1,\dots , m_n\}$ for all positive integers $n$. Here $(\sigma_{n_k})$ is a sequence of certain functions and $\beta_{\alpha_{n_k}, n_k}, p_{\alpha_{n_k}, n_k}$ are some real numbers. Note that a given numeral system can be with a finite or infinite, constant or  removable (or variable when $\alpha_i\in A_i\ne A_j\ni \alpha_j$ for some $i\ne j$)
alphabet. 

So, these problems introduce the problem on functional equations and systems of functional equations with several variables, on functional equations and systems of functional equations with  compositions of functions. 

In addition, one can consider expansions of functions and numbers by  complicated compositions of functions:
$$
\ldots \circ f_n \left( g^{(n)} _{m_n} \circ \ldots \circ g^{(n)} _{1}(x)\right)\circ \ldots \circ f_2 \left(g^{(2)} _{m_2} \circ \ldots \circ g^{(2)} _{1}(x)\right)\circ f_1 \left(g^{(1)} _{m_1} \circ \ldots \circ g^{(1)} _{1}(x)\right).
$$
Here $g^{(n)} _{m_n}, f_n $ are certain functions. One can consider partial cases of the last-mentioned complicated composition and the case when $m_n=\infty$. 

The next articles of the author of the present article will be devoted to such investigations.
\end{remark}

The second map considered in this article is useful for modeling fractals in space~$\mathbb R^2$. That is, the map
$$
f: x=\Delta^{-q} _{\underbrace{u\ldots u}_{\alpha_1-1}\alpha_1\underbrace{u\ldots u}_{\alpha_2-1}\alpha_2\ldots \underbrace{u\ldots u}_{\alpha_n-1}\alpha_n\ldots} \longrightarrow \Delta^{-q} _{\alpha_1\alpha_2...\alpha_n...},
$$
where $u\in \{0,1, \dots , q-1\}$ is a fixed number,  $\alpha_n \in\{1,2, \dots , q-1\}\setminus\{u\}$, and $3<q$ is a fixed positive integer, models a certain fractal in $\mathbb R^2$.

\section{One function defined in terms of alternating Cantor series}

Let us consider the function
$$
f(x)=f\left(\Delta^{-Q} _{\varepsilon_1\varepsilon_2... \varepsilon_n...}\right)=f\left(\sum^{\infty} _{n=1}{\frac{(-1)^n\varepsilon_n}{q_1q_2\cdots q_n}}\right)=\sum^{\infty} _{n=1}{\frac{\varepsilon_n}{(-q)^n}}=\Delta^{-q} _{\varepsilon_1\varepsilon_2... \varepsilon_n...}=y,
$$
where $\varepsilon_n\in\Theta_n$ and the condition $q_n\le q$ holds for all positive integers $n$.

\begin{lemma}
The function $f$ has the following properties:
\begin{enumerate}
\item $D(f)=[a_0-1,a_0]$, where $D(f)$ is the domain of definition of $f$ and
$$
a_0=\sum^{\infty} _{k=1}{\frac{q_{2k}-1}{q_1q_2\cdots q_{2k}}}=\sum^{\infty} _{k=1}{\frac{(-1)^{k+1}}{q_1q_2\cdots q_k}}, ~~~a_0-1=-\sum^{\infty} _{k=1}{\frac{q_{2k-1}-1}{q_1q_2\cdots q_{2k-1}}};
$$

\item If $E(f)$ is the range of values of $f$, then:
\begin{itemize}
\item $E(f)=[-\frac{q}{q+1},\frac{1}{q+1}]$ whenever the condition $q_n=q$ holds for all positive integers $n$,

\item $E(f)=[-\frac{q}{q+1},\frac{1}{q+1}]\setminus C_f$, where $C_f=C_1\cup C_2$,
$$
C_1=\left\{y: y=\Delta^{-q} _{\varepsilon_1\varepsilon_2... \varepsilon_{n}}, \varepsilon_n\notin\{q_n,q_n+1,\dots , q-1\}\text{for all $n$ such that $q_n<q$}\right\}
$$
and
$$
C_2=\left\{y: y=\Delta^q _{\varepsilon_1\varepsilon_2... \varepsilon_{n-1}[\varepsilon_{n}-1]0[q_{n+2}-1]0[q_{n+4}-1]0[q_{n+6}-1]...}\right\};
$$
\end{itemize}

\item $f(x)+f(a_0-x)=f(a_0)\le 1$;

\item $f\left(\sigma^k(x)\right)=\sigma^k\left(f(x)\right)$ for any $k\in \mathbb N$.
\end{enumerate}
\end{lemma}
\begin{proof}
\emph{The first property} follows from the definition of $f$. 

\emph{The second property} follows from the definition of $f$ and Theorem \ref{th: 3} (the next theorem).

Let us prove \emph{the third property}. Since
$$
a_0-x=\sum^{\infty} _{k=1}{\frac{q_{2k}-1-\varepsilon_{2k}}{q_1q_2\cdots q_{2k}}}+\sum^{\infty} _{k=1}{\frac{\varepsilon_{2k-1}}{q_1q_2\cdots q_{2k-1}}},
$$
we have
$$
f(a_0-x)=\sum^{\infty} _{k=1}{\frac{q_{2k}-1-\varepsilon_{2k}}{q^{2k}}}+\sum^{\infty} _{k=1}{\frac{\varepsilon_{2k-1}}{q^{2k-1}}}.
$$
Whence,
$$
f(x)+f(a_0-x)=\sum^{\infty} _{k=1}{\frac{\varepsilon_k}{(-q)^k}}+\sum^{\infty} _{k=1}{\frac{q_{2k}-1-\varepsilon_{2k}}{q^{2k}}}+\sum^{\infty} _{k=1}{\frac{\varepsilon_{2k-1}}{q^{2k-1}}}=\sum^{\infty} _{k=1}{\frac{q_{2k}-1}{q^{2k}}}=f(a_0)\le 1.
$$
Note that the last inequality is an equality whenever $y=x$, i.e., when the condition $q_n=q$ holds for all positive integers $n$.

Let us prove \emph{the fourth property}. We have
$$
f\left(\sigma^k(x)\right)=f\left(\sum^{\infty} _{j=k+1}{\frac{(-1)^{j-k}\varepsilon_{j}}{q_{k+1}q_{k+2}\cdots q_{j}}}\right)=\sum^{\infty} _{j=k+1}{\frac{\varepsilon_{j}}{(-q)^{j-k}}}=\sigma^k\left(\sum^{\infty} _{n=1}{\frac{\varepsilon_{n}}{(-q)^n}}\right)=\sigma^k\left(f(x)\right).
$$
\end{proof}

\begin{theorem} The following properties are true:
\label{th: 3}
\begin{itemize}
\item The function $f$ is continuous at nega-$Q$-irrational points from $[a_0-1,a_0]$.

\item The function $f$ is continuous at all nega-$Q$-rational points from $[a_0-1,a_0]$ if the condition $q_n=q$ holds for all positive integers $n$.

\item If there exist positive integers $n$ such that $q_n<q$, then points of the form
$$
\Delta^{-Q} _{\varepsilon_1\ldots\varepsilon_{m-1}\varepsilon_m[q_{m+1}-1]0[q_{m+3}-1]0[q_{m+5}-1]\ldots}=\Delta^Q _{\varepsilon_1\ldots\varepsilon_{m-1}[\varepsilon_m-1]0[q_{m+2}-1]0[q_{m+4}-1]0[q_{m+6}-1]\ldots},
$$
where $m<n$, are points of discontinuity of the function.
\end{itemize}
\end{theorem}
\begin{proof} Since for any $x\in [a_0-1,a_0]$ the equality 
$$
x=\Delta^{-Q} _{c_1c_2...c_m...}=\bigcap^{\infty} _{m=1} {\Delta^{-Q} _{c_1c_2...c_m}},
$$
is true (see~\cite{S. Serbenyuk alternating Cantor series 2013}), where $\Delta^{-Q} _{c_1c_2...c_m}$ is a nega-$Q$-cylinder,
let us consider   $x,x_0\in \Delta^{-Q} _{c_1c_2...c_m}$. Here $x$ is an arbitrary number, $x_0$ is a nega-$Q$-irrational number. Then
$$
\left|f (x) - f (x_0)\right|=\left|\sum^{\infty} _{k=m+1} {\frac{\varepsilon_k (f(x))-\varepsilon_k (f(x_0))}{(-q)^k}}\right|\le \frac{1}{q^m}\left|\sum^{\infty} _{k=m+1} {\frac{q_k-1}{q^{k-m}}}\right|
$$
$$
\le\sum^{\infty} _{k=m+1} {\frac{q-1}{q^{k}}}=\frac{1}{q^{m}}\to 0 ~\mbox{as} ~m \to \infty.
$$
So,
$$
\lim_{x\to x_0}{f(x)}=f(x_0).
$$
That is, the function $f$  is continuous at nega-$Q$-irrational points. 

If $x_0=\Delta^{-Q} _{\varepsilon_1\varepsilon_2...\varepsilon_n...}$ is  a nega-$Q$-rational point, then
$$
x_0=x^{(1)} _0=\begin{cases}
\Delta^{-Q} _{\varepsilon_1\varepsilon_2...\varepsilon_{n-1}\varepsilon_n [q_{n+1}-1]0[q_{n+3}-1]0[q_{n+5}-1]...}&\text{if $n$ is even}\\
\Delta^{-Q} _{\varepsilon_1\varepsilon_2...\varepsilon_{n-1}[\varepsilon_n-1]0 [q_{n+2}-1]0[q_{n+4}-1]0[q_{n+6}-1]...}&\text{if $n$ is odd}
\end{cases}
$$
$$
=\begin{cases}
\Delta^{-Q} _{\varepsilon_1\varepsilon_2...\varepsilon_{n-1}\varepsilon_n [q_{n+1}-1]0[q_{n+3}-1]0[q_{n+5}-1]...}&\text{if $n$ is odd}\\
\Delta^{-Q} _{\varepsilon_1\varepsilon_2...\varepsilon_{n-1}[\varepsilon_n-1]0 [q_{n+2}-1]0[q_{n+4}-1]0[q_{n+6}-1]...}&\text{if $n$ is even}
\end{cases}=x^{(2)} _0.
$$
Using the technique  for the case of nega-$Q$-irrational points, we obtain the folllowing for nega-$Q$-rational points:
$$
\lim_{x\to x_0+0}{f(x)}=f(x^{(1)} _0)
~~~\text{and}~~~
\lim_{x\to x_0-0}{f(x)}=f(x^{(2)} _0).
$$
 Hence
$$
\Delta_f=\lim_{x \to x_0+0} {f}(x)-\lim_{x \to x_0-0} {f (x)}=\frac{1}{q^n}-\frac{1}{q^n}\sum^{\infty} _{k=1}{\frac{q_{k+n}-1}{q^k}}\ne 0
$$
whenever there exists at least one $q_{k+n}< q$.
\end{proof}

\begin{corollary}
The set of all points of discontinuity of the function $f$ is:
\begin{itemize}
\item the empty set whenever $q_n=q$ for all $n\in\mathbb N$.
\item a finite set whenever $q_n\ne q$ for a finite number of $n$.
\item an infinite set whenever there exists an infinite subsequence $(n_k)$ of positive integers such that $q_{n_k}\ne q$.
\end{itemize}
\end{corollary}

\begin{remark}
To reach that the function $f$ be well-defined on the set of nega-$Q$-rational numbers from $[a_0-1,a_0]$, we shall not consider the  representation 
$$
\Delta^{-Q} _{\varepsilon_1\varepsilon_2... \varepsilon_{n-1}[\varepsilon_{n}-1]0[q_{n+2}-1]0[q_{n+4}-1]0[q_{n+6}-1]...}.
$$
\end{remark}

\begin{lemma}
The function $f$ is a  strictly increasing function on the domain.
\end{lemma}
\begin{proof}
Suppose $x_1=\Delta^{-Q} _{\alpha_1\alpha_2...\alpha_n...}$ and $x_2=\Delta^{-Q} _{\varepsilon_1\varepsilon_2...\varepsilon_n...}$ such that $x_1<x_2$. Then there exists $n_0$ such that $\alpha_i=\varepsilon_i$ for $i=\overline{1,n_0-1}$ and 
$$
\begin{cases}
\alpha_{n_0}<\varepsilon_{n_0} &\text{if $n_0$ is even}\\
\alpha_{n_0}>\varepsilon_{n_0} &\text{if $n_0$ is odd.}
\end{cases}
$$
Whence,
$$
f(x_2)-f(x_1)=\frac{\varepsilon_{n_0}-\alpha_{n_0}}{(-q)^{n_0}}+\sum^{\infty} _{j=n_0+1}{\frac{\varepsilon_j-\alpha_j}{(-q)^j}}.
$$

If $n_0$ is even, then 
$$
f(x_2)-f(x_1)\ge\frac{\varepsilon_{n_0}-\alpha_{n_0}}{q^{n_0}}-\sum^{\infty} _{k=1}{\frac{q_{n_0+2k-1}-1}{q^{n_0+2k-1}}}\ge \frac{1}{q^{n_0}}-\frac{1}{q^{n_0}}\sum^{\infty} _{k=1}{\frac{q_{n_0+2k-1}-1}{q^{2k-1}}}
$$
$$
\ge \frac{1}{q^{n_0}}-\frac{1}{q^{n_0}}\sum^{\infty} _{k=1}{\frac{q-1}{q^{2k-1}}}=\frac{1}{(q+1)q^{n_0}} >0, 
$$
since $n_0$ is even,  $\varepsilon_{n_0}-\alpha_{n_0}\ge 1$, and $q_n\le q$.

If $n_0$ is odd, then 
$$
f(x_2)-f(x_1)\ge-\frac{\varepsilon_{n_0}-\alpha_{n_0}}{q^{n_0}}-\sum^{\infty} _{k=1}{\frac{q_{n_0+2k}-1}{q^{n_0+2k}}}\ge \frac{\alpha_{n_0}-\varepsilon_{n_0}}{q^{n_0}}-\frac{1}{q^{n_0}}\sum^{\infty} _{k=1}{\frac{q_{n_0+2k}-1}{q^{2k}}}
$$
$$
\ge \frac{1}{q^{n_0}}-\frac{1}{q^{n_0}}\sum^{\infty} _{k=1}{\frac{q-1}{q^{2k}}}=\frac{1}{(q+1)q^{n_0-1}} >0, 
$$
since $n_0$ is odd,  $\alpha_{n_0}-\varepsilon_{n_0}\ge 1$, and $q_n\le q$.
\end{proof}

\begin{theorem} For the function $f$, the following statements are true:
\begin{itemize}
\item  If the condition $q_n=q$ holds for all positive integers $n$, then $f^{'} (x_0)=1$;
\item  If there exists an infinite sequence $(n_k)$ of positive integers such that $q_{n_k}<q$, then $f$ is a singular function;
\item  If there exists a finite sequence $(n_k)$ of positive integers such that $q_{n_k}<q$, then $f$ is a non-differentiable function.
\end{itemize}
\end{theorem}
\begin{proof}
Since~(see~\cite{S. Serbenyuk alternating Cantor series 2013})
for any nega-$Q$-cylinder $\Delta^{-Q} _{c_1c_2...c_m}$,  properties 
$$
\bigcap^{\infty} _{m=1} {\Delta^{-Q} _{c_1c_2...c_m}}=\Delta^{-Q} _{c_1c_2...c_m...}=x\in[a_0-1,a_0],
$$
$$
\Delta^{-Q} _{c_1c_2...c_m}=\begin{cases}
\left[\Delta^{-Q} _{c_1c_2...c_m[q_{m+1}-1]0[q_{m+3}-1]0[q_{m+5}-1]...}, \Delta^{-Q} _{c_1c_2...c_m0[q_{m+2}-1]0[q_{m+4}-1]0[q_{m+6}-1]...}\right]&\text{if $m$ is even}\\
\left[\Delta^{-Q} _{c_1c_2...c_m0[q_{m+2}-1]0[q_{m+4}-1]0[q_{m+6}-1]...},\Delta^{-Q} _{c_1c_2...c_m[q_{m+1}-1]0[q_{m+3}-1]0[q_{m+5}-1]...}\right]&\text{if $m$ is odd}
\end{cases},
$$
and
$$
\left|\Delta^{-Q} _{c_1c_2...c_m}\right|=\frac{1}{q_1q_2\cdots q_m}
$$
hold, we obtain 
$$
\mu_{f}\left(\Delta^{-Q} _{c_1c_2...c_n}\right)=f\left(\sup \Delta^{-Q} _{c_1c_2...c_n}\right)-f\left(\inf \Delta^{-Q} _{c_1c_2...c_n}\right)=\sum^{\infty} _{k=1}{\frac{q_{m+k}-1}{q^{m+k}}}
$$
and
$$
f^{'} (x_0)=\lim_{n\to\infty}{\frac{\mu_{f}{\left(\Delta^{-Q} _{c_1c_2...c_m}\right)}}{\left|\Delta^{-Q} _{c_1c_2...c_m}\right|}}=\lim_{m\to\infty}{\left(\frac{q_1q_2\cdots q_m }{q^m}\sum^{\infty} _{n=m+1}{\frac{q_n-1}{q^{n-m}}}\right)}
$$
  for $x_0\in\Delta^{-Q} _{c_1c_2...c_m}$.
Also, since $2\le q_m\le q$ holds for all positive integers $m$, we have
$$
\frac{1}{q-1}\lim_{m\to\infty}{\left(\frac{q_1q_2\cdots q_m }{q^m}\right)}\le\lim_{m\to\infty}{\left(\frac{q_1q_2\cdots q_m }{q^m}\sum^{\infty} _{n=m+1}{\frac{q_j-1}{q^{n-m}}}\right)}\le\lim_{m\to\infty}{\left(\frac{q_1q_2\cdots q_m }{q^m}\right)}.
$$
This completes the proof.
\end{proof}

Let us consider the following  infinite system of functional equations 
\begin{equation}
\label{def: function eq. system}
f\left(\sigma^{k-1}(x)\right)=-\frac{\varepsilon_k}{q}-\frac{1}{q}f\left(\sigma^{k}(x)\right),
\end{equation}
where $k=1, 2, \dots$, $\sigma$ is the shift operator of the nega-$Q$-expansion (here $\sigma^0 (x)=x$), and $x=\Delta^{-Q} _{\varepsilon_1\varepsilon_2...\varepsilon_n...}$. 

\begin{lemma} 
The function $f$ is the unique solution of  infinite system \eqref{def: function eq. system}  of functional equations in the class of determined and bounded on $[a_0-1,a_0]$ functions.
\end{lemma}
\begin{proof} Really, for an arbitrary $ x=\Delta^{-Q} _{\varepsilon_1\varepsilon_2...\varepsilon_k...}$ from $[a_0-1,a_0]$, we have 
$$
f(x)=-\frac{\varepsilon_1}{q}-\frac{1}{q}f(\sigma(x))=
-\frac{\varepsilon_1}{q}-\frac{1}{q}\left(-\frac{\varepsilon_2}{q}-\frac{1}{q}f\left(\sigma^2(x)\right)\right)
$$
$$
=-\frac{\varepsilon_1}{q}+\frac{\varepsilon_2}{q^2}+\frac{1}{q^2}\left(-\frac{\varepsilon_3}{q}-\frac{1}{q}f\left(\sigma^3(x)\right)\right)=\dots =\sum^{k} _{n=1}{\frac{\varepsilon_n}{(-q)^n}}+\frac{1}{(-q)^k}f\left(\sigma^{k}(x)\right)=\dots .
$$
Whence,
$$
f(x)=\lim_{k\to\infty}{\left(\sum^{k} _{n=1}{\frac{\varepsilon_n}{(-q)^n}}+\frac{1}{(-q)^k}f\left(\sigma^{k}(x)\right)\right)}=\sum^{\infty} _{k=1}{\frac{\varepsilon_k}{(-q)^k}},
$$
since functions $f, \sigma^k$ (for any $k\in\mathbb N$) are determined and bounded on the domains, and also
$$
\frac{1}{(-q)^k}\le\frac{1}{q^k}\to 0 ~~~(k\to\infty).
$$
\end{proof}

Let us consider integral properties of $f$. One can use the last lemma,  relationships \eqref{eq: Cantor series 2} and \eqref{eq: Cantor series 3}, and definitions of the shift operator, of alternating Cantor series, and of  expansion \eqref{eq: nega-q}. 
\begin{theorem}
The Lebesgue integral of the function $f$ can be calculated by the
formula
$$
\left|\int_{[a_0-1,a_0]}{f(x)dx}\right|=\sum^{\infty} _{k=1}{\frac{q_k-1}{2q^k}}.
$$
\end{theorem}
\begin{proof}
Since cylinders $\Delta^{-Q} _{c_1c_2...c_m}$ are left-to-right situated when $m$ is even and are right-to-left situated when $m$ is odd, this property is true for cylinders $\Delta^{-q} _{c_1c_2...c_m}$, and 
$$
d\left(\sigma^{k-1} (x)\right)=-\frac{1}{q_k}d\left(\sigma^{k} (x)\right),
$$
we obtain
$$
I=\int_{[a_0-1,a_0]}{f(x)dx}=\lim_{k\to\infty}{\left(-\sum^{k} _{n=1}{\frac{q_n-1}{2q^n}}+\frac{1}{q^k}I_k\right)}
$$
$$
=\lim_{k\to\infty}{\left(-\sum^{k} _{n=1}{\frac{q_n-1}{2q^n}}+\frac{1}{q^k}\int^{\sup_{x}\sigma^k(x)} _{\inf_{x}\sigma^k(x)}{f(\sigma^k(x))d(\sigma^k(x))}\right)}=-\sum^{k} _{n=1}{\frac{q_n-1}{2q^n}},
$$
where:
$$
I_1=\sum^{q_1-1} _{c_1=0}{\int^{\sup\Delta^{-Q} _{c_1}} _{\inf \Delta^{-Q} _{c_1}}{f(x)dx}} 
=\sum^{q_1-1} _{c_1=0}{\int^{\sup\Delta^{-Q} _{c_1}} _{\inf \Delta^{-Q} _{c_1}}{\left(-\frac{c_1}{q}-\frac{1}{q}f(\sigma(x))\right)dx}}
$$
$$
=-\sum^{q_1-1} _{i=0}{\frac{i}{qq_1}}+\frac{q_1}{qq_1}\int^{\sup_{x}\sigma(x)} _{\inf_{x}\sigma(x)}{f(\sigma(x))d(\sigma(x))}=-\frac{q_1-1}{2q}+\frac{1}{q}\int^{\sup_{x}\sigma(x)} _{\inf_{x}\sigma(x)}{f(\sigma(x))d(\sigma(x))},
$$
since $\left|\Delta^{-Q} _{c_1}\right|=\frac{1}{q_1}$;
$$
\frac{1}{q}\int^{\sup_{x}\sigma(x)} _{\inf_{x}\sigma(x)}{f(\sigma(x))d(\sigma(x))}=\frac{1}{q}I_2=\frac{1}{q}\sum^{q_2-1} _{c_2=0}{\int^{\sup\Delta^{-Q^{'}} _{c_2}} _{\inf \Delta^{-Q^{'}} _{c_2}}{f(\sigma(x))d(\sigma(x))}} 
$$
$$
=\frac{1}{q}\sum^{q_2-1} _{c_2=0}{\int^{\sup\Delta^{-Q^{'}} _{c_2}} _{\inf \Delta^{-Q^{'}} _{c_2}}{\left(-\frac{c_2}{q}-\frac{1}{q}f(\sigma^2(x))\right)d(\sigma(x))}}
$$
$$
=\frac{1}{q}\left(-\sum^{q_2-1} _{i=0}{\frac{i}{qq_2}}+\frac{q_2}{qq_2}\int^{\sup_{x}\sigma(x)} _{\inf_{x}\sigma(x)}{f(\sigma^2(x))d(\sigma^2(x))}\right)=-\frac{q_2-1}{2q^2}+\frac{1}{q^2}\int^{\sup_{x}\sigma^2(x)} _{\inf_{x}\sigma^2(x)}{f(\sigma^2(x))d(\sigma^2(x))},
$$
since $\left|\Delta^{-Q^{'}} _{c_2}\right|=\frac{1}{q_2}$ and 
$$
x=\Delta^{-Q^{'}} _{\varepsilon_2\varepsilon_3...}=\sum^{\infty} _{k=2}{\frac{(-1)^{k-1}\varepsilon_k}{q_2q_3\cdots q_k}};
$$
$$
\dots \dots \dots \dots \dots \dots \dots \dots \dots \dots \dots \dots \dots \dots \dots \dots \dots \dots 
$$
$$
\frac{1}{q^{k-1}}\int^{\sup_{x}\sigma^{k-1}(x)} _{\inf_{x}\sigma^{k-1}(x)}{f(\sigma^{k-1}(x))d(\sigma^{k-1}(x))}=\frac{1}{q^{k-1}}I_k=\frac{1}{q^{k-1}}\sum^{q_k-1} _{c_k=0}{\int^{\sup\Delta^{-Q^{(k-1)}} _{c_k}} _{\inf \Delta^{-Q^{(k-1)}} _{c_k}}{f(\sigma^{k-1}(x))d(\sigma^{k-1}(x))}} 
$$
$$
=\frac{1}{q^{k-1}}\sum^{q_k-1} _{c_k=0}{\int^{\sup\Delta^{-Q^{{(k-1)}}} _{c_k}} _{\inf \Delta^{-Q^{(k-1)}} _{c_k}}{\left(-\frac{c_k}{q}-\frac{1}{q}f(\sigma^k(x))\right)d(\sigma^{k-1}(x))}}
$$
$$
=\frac{1}{q^{k-1}}\left(-\sum^{q_k-1} _{i=0}{\frac{i}{qq_k}}+\frac{q_k}{qq_k}\int^{\sup_{x}\sigma^{k-1}(x)} _{\inf_{x}\sigma^{k-1}(x)}{f(\sigma^{k}(x))d(\sigma^{k}(x))}\right)=-\frac{q_k-1}{2q^k}+\frac{1}{q^k}\int^{\sup_{x}\sigma^k(x)} _{\inf_{x}\sigma^k(x)}{f(\sigma^k(x))d(\sigma^k(x))},
$$
since $\left|\Delta^{-Q^{(k-1)}} _{c_k}\right|=\frac{1}{q_k}$ and 
$$
x=\Delta^{-Q^{(k-1)}} _{\varepsilon_k\varepsilon_{k+1}...}=\sum^{\infty} _{n=k}{\frac{(-1)^{n-k-1}\varepsilon_n}{q_kq_{k+1}\cdots q_n}}.
$$
This completes the proof.
\end{proof}

\section{Some fractals  defined  in terms of  certain maps in $\mathbb R^2$}

Let us consider the following function
$$
h: x=\Delta^{-q} _{\underbrace{u\ldots u}_{\alpha_1-1}\alpha_1\underbrace{u\ldots u}_{\alpha_2-1}\alpha_2\ldots \underbrace{u\ldots u}_{\alpha_n-1}\alpha_n\ldots} \longrightarrow \Delta^{-q} _{\alpha_1\alpha_2...\alpha_n...},
$$
where $u\in \{0,1, \dots , q-1\}$ is a fixed number,  $\alpha_n \in \Theta=\{1,2, \dots , q-1\}\setminus\{u\}$, and $3<q$ is a fixed positive integers. 
This function can be represented by the following form.
$$
h: x=-\frac{u}{q+1}+\sum^{\infty} _{n=1}{\frac{\alpha_n-u}{(-q)^{\alpha_1+\alpha_2+\dots +\alpha_n}}}
 \longrightarrow \sum^{\infty} _{n=1}{\frac{\alpha_n}{(-q)^{n}}}=h(x)=y.
$$

\begin{theorem}
The function $h$ has the following properties:
\begin{enumerate}
\item The domain of definition $D(h)$ of the function $h$ is a set having the following properties:
\begin{itemize}
\item $D(h)$ is an uncountable,   perfect,   and nowhere dense set;
\item the Lebesgue measure of $D(h)$ equals zero;
\item    $D(h)$  is a self-similar fractal whose Hausdorff dimension $\alpha_0$ satisfies the  equation 
$$
\sum _{p \ne u, p \in \{1,2,\dots , q-1\}} {\left(\frac{1}{q}\right)^{p \alpha_0}}=1.
$$
\end{itemize}

\item The range of values $E(h)$ of $h$ is a self-similar fractal 
$$
E(h)=\{y: y=\Delta^{-q} _{\alpha_1\alpha_2...\alpha_n...}, \alpha_n\in\Theta\}
$$
for  which  the Hausdorff dimension $\alpha_0$ equals $\log_q {|\Theta|},$ where $|\cdot|$ is the number of elements of a set.

\item The function $h$  is well defined and is a bijective mapping on the domain.

\item The function $h$ is continuous at any point on the domain.

\item On the domain of definition the function $h$ is:
\begin{itemize}
\item decreasing whenever $u\in\{0,1\}$ for all $q>3$;
\item increasing  whenever $u\in \{q-2,q-1\}$ for all $q>3$;
\item  not monotonic  whenever $u\in\{2,3, \dots , q-3\}$ and $q>4$.
\end{itemize}

\item The function $h$ is non-differentiable on the domain.

\item The following relationship is true for any positive integer $n$:
$$
h\left(\sigma^{\alpha_1+\alpha_2+\dots +\alpha_n}(x)\right)=\sigma^n(h(x)).
$$
Here $\sigma$ is the shift operator.

\item The function does not preserve the Hausdorff dimension.
\end{enumerate}
\end{theorem}
\begin{proof}
The domain  $D(h)$ of the function $h$ is   a set whose elements represented in terms of nega-$q$-ary representation~\eqref{eq: nega-q}. Representations of elements of $D(h)$ contain only  combinations of digits from a some subset of the following set (it depends on  fixed papameters $u$ and $q$):
$$
\left\{1, u2, uu3, \dots, \underbrace{u\ldots u}_{u-2}[u-1], \underbrace{u\ldots u}_{u}[u+1], \dots , \underbrace{u\ldots u}_{q-2}[q-1] \right\}.
$$
That is (see \cite{ {S. Serbenyuk 2017  fractals}, {S. Serbenyuk   fractals}}), such set is a self-similar fractal whose Hausdorff dimension $\alpha_0$ satisfies the following equation 
$$
\sum _{p \ne u, p \in \{1,2,\dots , q-1\}} {\left(\frac{1}{q}\right)^{p \alpha_0}}=1.
$$
Hence this set is an uncountable,   perfect,  and  nowhere dense set of zero Lebesgue measure. 

\emph{The second property} follows from the definition of $h$.  

\emph{Property~8} follows from the first and second properties, since $\alpha_0(D(h))\ne \alpha_0(E(h))$, where $\alpha_0(\cdot)$ is the Hausdorff dimension of a set.

Let us prove \emph {Property~3 and Property~4.} The set $D(h)$ does not contain numbers with zeros in own nega-$q$-representations whenever $u>0$. If $u=0$, then  $D(h)$ does not contain numbers having a period $(0[q-1]))$ or $([q-1]0)$  in own nega-$q$-representations. Whence, $D(h)$ does contain nega-$q$-rational numbers i.e., numbers of the form
$$
\Delta^{-q} _{\alpha_1\alpha_2...\alpha_{n-1}[\alpha_{n}-1](0[q-1])}=\Delta^{-q} _{\alpha_1\alpha_2...\alpha_{n-1}\alpha_{n}([q-1]0)}.
$$
That is, any element of $D(h)$  has the unique nega-$q$-representation. Therefore the condition $h(x_1)\ne h(x_2)$ holds for $x_1\ne~x_2$. Let us note that a value $h(x)\in E(h)$ is assigned to an arbitrary $x\in D(h)$ and  vice versa.

Let us consider a nega-$q$-cylinder $\Delta^{-q} _{\underbrace{u\ldots u}_{c_1-1}c_1\underbrace{u\ldots u}_{c_2-1}c_2\ldots \underbrace{u\ldots u}_{c_n-1}c_n}$. Since
$$
D(h)\ni x=\bigcap^{\infty} _{n=1}{\Delta^{-q} _{\underbrace{u\ldots u}_{c_1-1}c_1\underbrace{u\ldots u}_{c_2-1}c_2\ldots \underbrace{u\ldots u}_{c_n-1}c_n}},
$$
let us consider $x,x_0\in \left(D(h) \cap\Delta^{-q} _{\underbrace{u\ldots u}_{c_1-1}c_1\underbrace{u\ldots u}_{c_2-1}c_2\ldots \underbrace{u\ldots u}_{c_n-1}c_n} \right)$. Then 
$$
|h(x)-h(x_0)|=|h\left(\Delta^{-q} _{\underbrace{u\ldots u}_{c_1-1}c_1\ldots \underbrace{u\ldots u}_{c_n-1}c_n\underbrace{u\ldots u}_{\alpha_{n+1}-1}\alpha_{n+1}(x)\underbrace{u\ldots u}_{\alpha_{n+2}-1}\alpha_{n+2}(x)...}\right)
$$
$$
-h\left(- \Delta^{-q} _{\underbrace{u\ldots u}_{c_1-1}c_1\ldots \underbrace{u\ldots u}_{c_n-1}c_n\underbrace{u\ldots u}_{\alpha_{n+1}-1}\alpha_{n+1}(x_0)\underbrace{u\ldots u}_{\alpha_{n+2}-1}\alpha_{n+2}(x_0)...}\right)|
$$
$$
\le \left|\sum^{\infty} _{k=n+1}{\frac{\alpha_k(x)-\alpha_k(x_0)}{(-q)^k}}\right|< \frac{1}{q^{n}} \to 0 ~~~~~(n\to \infty).
$$
So, for any $x_0\in D(h)$, the following holds
$$
\lim_{x\to x_0}{|h(x)-h(x_0)|}=0.
$$

Let us prove \emph{Property~5}. Consider $x_1,x_2\in D(h)$, $x_1\ne x_2$, i.e., 
$$
x_1=\Delta^{-q} _{\underbrace{u\ldots u}_{\alpha_1-1}\alpha_1\underbrace{u\ldots u}_{\alpha_2-1}\alpha_2\ldots \underbrace{u\ldots u}_{\alpha_{k_0-1}-1}\alpha_{k_0-1}\underbrace{u\ldots u}_{\alpha_{k_0}-1}\alpha_{k_0}\underbrace{u\ldots u}_{\alpha_{k_0+1}-1}\alpha_{k_0+1}\ldots}
$$
and
$$
x_2=\Delta^{-q} _{\underbrace{u\ldots u}_{\gamma_1-1}\gamma_1\underbrace{u\ldots u}_{\gamma_2-1}\gamma_2\ldots \underbrace{u\ldots u}_{\gamma_{k_0-1}-1}\gamma_{k_0-1}\underbrace{u\ldots u}_{\gamma_{k_0}-1}\gamma_{k_0}\underbrace{u\ldots u}_{\gamma_{k_0+1}-1}\gamma_{k_0+1}\ldots},
$$
where $\alpha_i=\gamma_i$ for $i=\overline{1,k_0-1}$ and $\alpha_{k_0}\ne\gamma_{k_0}$. Whence, 
$$
y_1=h(x_1)=\Delta^{-q} _{\alpha_1\alpha_2\ldots\alpha_{k_0-1}\alpha_{k_0}\alpha_{k_0+1}\ldots}, ~~~ y_2=h(x_2)=\Delta^{-q} _{\gamma_1\gamma_2\ldots \gamma_{k_0-1}\gamma_{k_0}\gamma_{k_0+1}\ldots},
$$
and $y_1<y_2$ whenever the following system of conditions holds:
\begin{equation}
\label{eq: system conditions 1}
\begin{cases}
\alpha_{k_0}<\gamma_{k_0} &\text{if $k_0$ is even}\\
\alpha_{k_0}>\gamma_{k_0} &\text{if $k_0$ is odd.}
\end{cases}
\end{equation}

Let us consider $D(h)$ more detail.

Suppose $u=0$. Then 
$$
x_1=\Delta^{-q} _{\underbrace{0\ldots 0}_{\alpha_1-1}\alpha_1\underbrace{0\ldots 0}_{\alpha_2-1}\alpha_2\ldots \underbrace{0\ldots 0}_{\alpha_{k_0-1}-1}\alpha_{k_0-1}\underbrace{0\ldots 0}_{\alpha_{k_0}-1}\alpha_{k_0}\underbrace{0\ldots 0}_{\alpha_{k_0+1}-1}\alpha_{k_0+1}\ldots}=\sum^{\infty} _{k=1}{\frac{\alpha_k}{(-q)^{\alpha_1+\alpha_2+\dots +\alpha_k}}}
$$
and
$$
x_2=\Delta^{-q} _{\underbrace{0\ldots 0}_{\gamma_1-1}\gamma_1\underbrace{0\ldots 0}_{\gamma_2-1}\gamma_2\ldots \underbrace{0\ldots 0}_{\gamma_{k_0-1}-1}\gamma_{k_0-1}\underbrace{0\ldots 0}_{\gamma_{k_0}-1}\gamma_{k_0}\underbrace{0\ldots 0}_{\gamma_{k_0+1}-1}\gamma_{k_0+1}\ldots}=\sum^{\infty} _{k=1}{\frac{\gamma_k}{(-q)^{\gamma_1+\gamma_2+\dots +\gamma_k}}}.
$$

Suppose that $\tau_{k_0}:=\min\{\alpha_1+\alpha_2+\dots +\alpha_{k_0}, \gamma_1+\gamma_2+\dots +\gamma_{k_0}\}$.

Since, for $\Delta^{-q} _{\underbrace{0\ldots 0}_{\alpha-1}\alpha(0)}<\Delta^{-q} _{\underbrace{0\ldots 0}_{\beta-1}\beta(0)}$, the system of conditions 
$$
\begin{cases}
\Delta^{-q} _{\underbrace{000000\ldots 0000000}_{\alpha-1}\alpha(0)}<\Delta^{-q} _{\underbrace{000\ldots 00}_{\beta-1}\beta(0)} &\text{if $\min\{\alpha,\beta\}$ is even and $\alpha>\beta$}\\
\Delta^{-q} _{\underbrace{000\ldots 00}_{\alpha-1}\alpha(0)}<\Delta^{-q} _{\underbrace{00000\ldots 0000000}_{\beta-1}\beta(0)} &\text{if $\min\{\alpha,\beta\}$ is odd and $\alpha<\beta$}
\end{cases}
$$
is true, we have that $x_1<x_2$ whenever the  following system of conditions holds:
\begin{equation}
\label{eq: system conditions 2}
\begin{cases}
\alpha_{k_0}>\gamma_{k_0} &\text{if $\tau_{k_0}$ is even}\\
\alpha_{k_0}<\gamma_{k_0} &\text{if $\tau_{k_0}$ is odd.}
\end{cases}
\end{equation}
However, since 
$$
\tau_{k_0}\equiv \sum^{k_0} _{j=1}{i^{'} _{j} (\mod 2)}\equiv \tau^{'} _{k_0} (\mod 2),
$$
 where $i^{'} _{j}, \tau^{'} _{k_0}\in\{0,1\}$, we obtain $\tau^{'} _{k_0}\le k^{'} _0\equiv k_0 (\mod 2)$. Here $k^{'} _0\in\{0,1\}$. 

So, if $k_0$ is even, then $\tau_{k_0}$ is even and from \eqref{eq: system conditions 1}, 
\eqref{eq: system conditions 2} it follows that $y_2>y_1$ for $x_1<x_2$, i.e., the function $h$ is decreasing. 

By analogy, if  $k_0$ is odd, then $\tau_{k_0}$ is even or is odd but $h$ is a decreasing function.

Let us remark that $h$ is decreasing for the case when $u=1$ because corresponding  considerations for $u=0$ and $u=1$ are identical. 

Let $u\in\{2,3, \dots , q-3\}, q>4$. Then $\Delta^{-q} _{\underbrace{u\ldots u}_{\alpha-1}\alpha(0)}<\Delta^{-q} _{\underbrace{u\ldots u}_{\beta-1}\beta(0)}$ whenever one of the following cases holds:
\begin{itemize}
\item $\min\{\alpha, \beta\}$ is odd, $\alpha>\beta$, and $u>\beta$;
\item $\min\{\alpha, \beta\}$ is odd, $\alpha<\beta$, and $u<\alpha$;
\item $\min\{\alpha, \beta\}$ is even, $\alpha<\beta$, and $u>\alpha$;
\item $\min\{\alpha, \beta\}$ is even, $\alpha>\beta$, and $u<\beta$.
\end{itemize}
Note that exist $\alpha_n$ such that $\alpha_n<u$ and  also $\alpha_n>u$  in our case (when $u\in\{2,3, \dots , q-3\}, q>4$). For example:
\begin{itemize}
\item if $u=2$, then $\alpha_n=1<2$ and $2<\alpha_n\in\{3,4,\dots , q-1\}$;
\item  if $u=3$, then $\{1,2\}\ni \alpha_n<3$ and $3<\alpha_n\in\{4,5, \dots , q-1\}$;
\item  if $u=q-3$, then $\{1,2, \dots , q-4\}\ni \alpha_n<u<\alpha_n\in\{q-2, q-1\}$.
\end{itemize}
So, $h$ is not monotonic in this case.

Suppose $u\in \{q-2, q-1\}$. Then  $\Delta^{-q} _{\underbrace{u\ldots u}_{\alpha-1}\alpha(0)}<\Delta^{-q} _{\underbrace{u\ldots u}_{\beta-1}\beta(0)}$ whenever one of the following holds:
\begin{itemize}
\item  $\min\{\alpha, \beta\}$ is odd and $\alpha>\beta$;
\item  $\min\{\alpha, \beta\}$ is even and  $\alpha<\beta$.
\end{itemize}
Hence $x_1<x_2$ whenever
$$
\begin{cases}
\alpha_{k_0}<\gamma_{k_0} &\text{if $\tau_{k_0}$ is even}\\
\alpha_{k_0}>\gamma_{k_0} &\text{if $\tau_{k_0}$ is odd}.
\end{cases}
$$
Using  \eqref{eq: system conditions 1}, $k^{'} _0$, and $\tau^{'} _{k_0}$, we obtain that 
\begin{itemize}
\item if $k_0$ is even, then $\tau_{k_0}$ is even and $h$ is an increasing function;
\item if $k_0$ is odd, then $h$ is an increasing function for the cases of even and odd $\tau_{k_0}$.
\end{itemize}

Let us prove \emph{the 6th property}.  By analogy with similar investigations for positive expansions (see~\cite{Serbenyuk 2019 functions}), we obtain the following.  Let $\Delta^{-q} _{\underbrace{u\ldots u}_{c_1-1}c_1\underbrace{u\ldots u}_{c_2-1}c_2\ldots \underbrace{u\ldots u}_{c_{n-1}-1}c_{n-1}}$ be an arbitrary cylinder. Then let us consider a sequence $(x_n)$ of numbers 
$$
x_n=\Delta^{-q} _{\underbrace{u\ldots u}_{c_1-1}c_1\underbrace{u\ldots u}_{c_2-1}c_2\ldots\underbrace{u\ldots u}_{c_{n-1}-1}c_{{n} -1}  \underbrace{u\ldots u}_{\alpha_{n} -1}\alpha_{n}\underbrace{u\ldots u}_{\alpha_{n+1}-1}\alpha_{n+1}\ldots}
$$
 and a fixed number  $x_0=\Delta^{-q} _{\underbrace{u\ldots u}_{c_1-1}c_1\underbrace{u\ldots u}_{c_2-1}c_2\ldots\underbrace{u\ldots u}_{c_{n-1}-1}c_{{n}-1}  \underbrace{u\ldots u}_{c-1}c\underbrace{u\ldots u}_{\alpha_{n+1}-1}\alpha_{n+1}\ldots}$, where $c$ is a fixed number. Then
$$
\lim_{x\to x_0}{\frac{h(x)-h(x_0)}{x-x_0}}=\lim_{x\to x_0}{\frac{\frac{\alpha_n-c}{(-q)^n}}{\frac{\alpha_n}{(-q)^{c_1+c_2+\dots +c_{n-1}+\alpha_n}}-\frac{c}{(-q)^{c_1+c_2+\dots +c_{n-1}+c}}}}.
$$
Since conditions $x_n\to x_0$ and $n\to\infty$ are equivalent  and $\frac{\alpha_n-c}{\alpha_n(-q)^c-c(-q)^{\alpha_n}}$ is a some number, we get
$$
\lim_{x\to x_0}{\frac{h(x)-h(x_0)}{x-x_0}}=\lim_{n\to \infty}{\left(\frac{\alpha_n-c}{\alpha_n(-q)^c-c(-q)^{\alpha_n}}(-q)^{c_1+c_2+\dots +c_{n-1}+\alpha_n+c-n}\right)}=\pm \infty
$$
whenever there exists an infinite number of $c_n\ne 1$, because  $c_1+c_2+\dots +c_{n-1} \ge n-1$ and $c_1+c_2+\dots +c_{n-1} = n-1$ when $c_j=1$, $j=\overline{1,n-1}$.

If all  $c_j=1$, where $j=\overline{1,n-1}$, then 
$$
\lim_{x\to x_0}{\frac{h(x)-h(x_0)}{x-x_0}}=\lim_{n\to \infty}{\left(\frac{\alpha_n-c}{\alpha_n(-q)^c-c(-q)^{\alpha_n}}(-q)^{\alpha_n+c-1}\right)}.
$$
However, $\alpha_n, c$ are fixed different numbers. So, the function $h$ is non-differentiable.

\emph{Property 7}. It is easy to see that 
$$
h\left(\sigma^{\alpha_1+\alpha_2+\dots +\alpha_n}(x)\right)=h(\Delta^{-q} _{\underbrace{u\ldots u}_{\alpha_{n+1}-1}\alpha_{n+1}\ldots\underbrace{u\ldots u}_{\alpha_{n+2}-1}\alpha_{n+2}   \ldots})=\Delta^{-q} _{\alpha_{n+1}\alpha_{n+3}...}=\sigma^n(h(x)).
$$
\end{proof}

\begin{theorem}
The Hausdorff dimension of a graph of the function  $h$ is equal to $1$.
\end{theorem}
\begin{proof} 
Let us prove the statement by analogy with the similar proof for positive expansions~(\cite{Serbenyuk 2019 functions}; for example, for other functions,  such proof is given in~\cite{S. Serbenyuk functions with complicated local structure 2013}).

Suppose that 
$$
X=\left[-\frac{q}{q+1},\frac{1}{q+1}\right]\times\left[-\frac{q}{q+1},\frac{1}{q+1}\right]
$$
$$
=\left\{(x,y): x=\sum^{\infty} _{m=1} {\frac{\alpha_m}{(-q)^{m}}}, \alpha_{m} \in \Theta_q=\{0,1,\dots ,q-1\},
y=\sum^{\infty} _{m=1} {\frac{\beta_m}{(-q)^{m}}}, \beta_{m} \in \Theta_q \right\}.
$$
Then the set 
$$
\sqcap_{(\alpha_{1}\beta_{1})(\alpha_{2}\beta_{2})...(\alpha_{m}\beta_{m})}=\Delta^{-q} _{\alpha_{1}\alpha_{2}...\alpha_{m}}\times\Delta^{-q} _{\beta_{1}\beta_{2}...\beta_{m}}
$$
is a square with a side length of $q^{-m}$. This square is called \emph{a square of rank $m$ with the base $(\alpha_{1}\beta_{1})(\alpha_{2}\beta_{2})\ldots (\alpha_{m}\beta_{m})$}.

It is known that if  $E\subset X$, then the number
$$
\alpha^{K}(E)=\inf\{\alpha: \widehat{H}_{\alpha} (E)=0\}=\sup\{\alpha: \widehat{H}_{\alpha} (E)=\infty\},
$$
where
$$
\widehat{H}_{\alpha} (E)=\lim_{\varepsilon \to 0} \left[{\inf_{d\leq \varepsilon} {K(E,d)d^{\alpha}}}\right]
$$
and $K(E,d)$ is the minimum number of squares of diameter $d$ required to cover the set $E$, is called \emph{ the fractal cell entropy dimension of the set E.} It is easy to see that $\alpha^{K}(E)\ge \alpha_0(E)$.

From the definition and properties of the function  $g$ it follows that the graph of the function belongs to $\tau=|\Theta|$ squares from  $q^2$
first-rank squares (here $\tau$ is equal to $(q-1)$ for $u=0$ and $\tau$ is equal to $(q-2)$ for $u\ne 0$):
$$
\sqcap_{(i_1i_1)}=\left[\Delta^{-q} _{\underbrace{u\ldots u}_{i_1-1}i_{1}}, \Delta^{-q} _{i_1}
\right],~i_1 \in \Theta_q.
$$

The graph of  the function $f$ belongs to $\tau^2$ squares from $q^4$ second-rank squares:
$$
\sqcap_{(i_1i_2)(i_1i_2)}=\left[\Delta^{-q} _{\underbrace{u\ldots u}_{i_1-1}i_{1}\underbrace{u\ldots u}_{i_2-1}i_{2}}, \Delta^{-q} _{i_1i_2}
\right],~i_1, i_2 \in \Theta_q.
$$

The graph $\Gamma_{g}$ of the function  $g$ belongs to $\tau^m$ squares of rank $m$ with sides $q^{\alpha_1+\alpha_2+\dots+\alpha_m}$ and $q^{-m}$. Then
$$
\widehat{H}_{\alpha} (\Gamma_g)=\lim_{\overline{m \to \infty}} {\tau^m \left(\sqrt{q^{-2(\alpha_1+\alpha_2+\dots+\alpha_m)}+q^{-2m}}\right)^{\alpha}}.
$$
Since $q^{-m(q-1)} \le q^{-(\alpha_1+\alpha_2+\dots+\alpha_m)}\le q^{-m}$, we get 
$$
\widehat{H}_{\alpha} (\Gamma_g)=\lim_{\overline{m \to \infty}} {\tau^m \left(2\cdot q^{-2m}\right)^{\frac{\alpha}{2}}}=\lim_{\overline{m \to \infty}} {\tau^m \left(2\cdot q^{-2m}\right)^{\frac{\alpha}{2}}}=\lim_{\overline{m \to \infty}}
{\left(2^{\frac{\alpha}{2}}\cdot \tau^m \cdot q^{-m\alpha}\right)}
$$
$$
=\lim_{\overline{m \to \infty}}{\left(2^{\frac{\alpha}{2}}\cdot \left(\frac{\tau}{q^{\alpha}}\right)^m\right)}
$$
for $\alpha_1+\alpha_2+\dots+\alpha_m=m$
and 
$$
\widehat{H}_{\alpha} (\Gamma_g)=\lim_{\overline{m \to \infty}} {\tau^m \left(q^{-2m(q-1)}+q^{-2m}\right)^{\frac{\alpha}{2}}}=\lim_{\overline{m \to \infty}}{\left(\left(\frac{\tau^{\frac{1}{\alpha}}}{q}\right)^{2m}+\left(q^{1-q}\tau^{\frac{1}{\alpha}}\right)^{2m}\right)^{\frac{\alpha}{2}}}
$$
for $\alpha_1+\alpha_2+\dots+\alpha_m=m(q-1)$.

It is obvious that if $\left(\frac{\tau}{q^{\alpha}}\right)^m\to 0$,  $\left(\frac{\tau^{\frac{1}{\alpha}}}{q}\right)^{2m}\to 0$,and $\left(q^{1-q}\tau^{\frac{1}{\alpha}}\right)^{2m}\to 0$
 for $\alpha >1$, and the graph of the function has self-similar properties, then  $\alpha^K (\Gamma_g)=\alpha_0(\Gamma_g)=~1$. 
\end{proof}

\end{document}